\documentclass[12pt,a4paper,reqno]{amsart}
\usepackage{amsmath,amssymb,amsthm,amscd,mathrsfs} 
\usepackage[T1]{fontenc}
\usepackage[latin1]{inputenc}
\usepackage[english]{babel}
\usepackage[mathscr]{eucal}
\usepackage{graphicx}
\usepackage{float}
\usepackage{tikz}
\usetikzlibrary{positioning}
\usetikzlibrary{shapes.geometric}
\usetikzlibrary{fit}
\usetikzlibrary{patterns}
\usepackage{caption}
\usepackage{tikz-cd}
\usepackage{cases}
\usepackage{hyperref}
\usepackage{chngcntr}

\usetikzlibrary{decorations.text}

\input{comment.sty}
\includecomment{FZ}
\includecomment{MM}
\counterwithin{equation}{section}

\renewcommand{\th}{\theta}

\newcommand{\ph}{\phi}
\def\ph{\phi}
\def\md#1{\ \mbox{\rm(mod }{#1})}
\def\nph#1{N_{\ph}(#1)}
\def\npp#1{N_{\ph}^+(#1)}
\def\ol{\overline}

\def\ph{\phi}

\newcommand{\Q}{{\mathbb Q}}
\newcommand{\Z}{{\mathbb Z}}

\newcommand{\F}{{\mathbb F}}

\def\fph{\mathbb{F}_{\ph}}
\def\md#1{\ \mbox{\rm(mod }{#1})}
\def\nph#1{N_{\ph}(#1)}
\def\npp#1{N_{\ph}^+(#1)}

\def\ol{\overline}

\newcommand{\al}{{\alpha}}

\newcommand{\p}{\mathfrak{p}}

\newtheorem{theorem}{Theorem}[section]

\newtheorem{lemma}[theorem]{Lemma}
\newtheorem{cor}[theorem]{Corollary}

\theoremstyle{definition}

\theoremstyle{remark}
\newtheorem{rem}[theorem]{Remark}

\begin{document}
\title[]{ On index divisors and monogenity of certain number fields defined by trinomials}
\textcolor[rgb]{1.00,0.00,0.00}{}
\author{Lhoussain El Fadil}\textcolor[rgb]{1.00,0.00,0.00}{}
\address{Faculty of Sciences Dhar El Mahraz, P.O. Box  1796 Atlas-Fes , Sidi mohamed ben Abdellah University,  Morocco}\email{lhoussain.elfadil@usmba.ac.ma,\,\, lhouelfadil2@gmail.com}
\date{March 14, 2022}
\maketitle
\vspace{0.3cm}
\begin{abstract}  
 Let  $K$ be a  number field generated  by a  root  $\th$ of a monic irreducible trinomial  $F(x) = x^n+ax^{m}+b \in \Z[x]$. 
In this paper, we study the problem of $K$.  More precisely, we provide some explicit conditions on $a$,  $b$, $n$, and $m$ for which $K$ is not monogenic. As  applications, we show that there are infinite families of non-monogenic number fields defined by trinomials of degree $n=2^r\cdot3^k$ with $r$ and $k$ two positive integers.  We also give   infinite families of non-monogenic sextic number fields defined by trinomials. Some  illustrating examples are giving at the end of this paper.  
\end{abstract}
\maketitle

\section{Introduction}
Let $K$ be a number field of degree $n$ with ring of integers $\Z_K$, and $d_K$ its absolute discriminant.
It is well-know that $\Z_K$ is a free Abelian group of rank $n$. Thanks to a classification theorem of free Abelian groups of finite rank, the  index $(\Z_K:\Z[\th])$ is finite and called the index of $\Z[\th]$.
The number field $K$ is called {\it monogenic} if $\Z_K$ admits a $\Z$-basis of type $(1,\al,\ldots,\al^{n-1})$ for some  $\al\in\Z_K$.
Monogenity of number fields is a classical problem of algebraic number theory, going back to Dedekind, Hasse, and Hensel (see for instance \cite{Ha,  He, G19}). 
For any integral primitive element $\al$ of $K$,
we denote by 
\[
ind(\al)=(\Z_K:\Z[\al])
\]
the index of $\al$.
As it is known \cite{G19}, we have the following index formula:
\[
|\triangle(\al)|=ind(\al)^2\cdot |d_K|
\]
where $\triangle(\al)$ is the discriminant of the minimal polynomial of  $\al$ over $\Q$.

It is clear  that   $ind(\al)=1$  for some integral primitive element $\al$ of $K$ if and only if $(1,\al,\ldots, \al^{n-1})$ is a power integral basis
of $\Z_K$. 

The problem of testing the monogenity of number fields and constructing generators of power integral bases have been intensively studied during the last four decades mainly by Ga\'al, Gy\"ory, Nakahara, pohst and their collaborators, see for instance \cite{AN, 13a, P}.
Especially a delicate and intensively studied problem is the monogenity of pure number fields. Recently, many authors are interested on monogenity of number fields defined by trinomials.
In \cite{AK, SK}, Khanduja et al. studied the integral closedness of some number fields defined by trinomials. Their results are refined by Ibarra et al. (see \cite{Smt}).  Remark that Khanduja's results given in \cite{AK, SK}, can only decide on the integral closedness of $\Z[\al]$, but cannot gave a complete answer to the monogenity problem of  $K$. In \cite{Lsc, Lac, LPh, LD}, Jones et al. introduced  a new notion monogenity, namely monogenity  of irreducible polynomials.  According to Jones  definition, if a polynomial $F(x)$ is monogenic, then $\Q(\al)$ is monogenic, but the converse is not true because  a number field generated by a  root  of a non monogenic polynomial can be monogenic. Therefore Jones's and Khanduja's results  cover partially the study of the monogenity of number fields.
Davis and Spearman \cite{DA} studied the index of quartic number fields $K$ generated by a root of 
such a quartic  trinomial $F(x)=x^4+ax+b\in\Z[x]$.
They gave necessary and sufficient conditions on $a$ and $b$, which charaterize when a prime $p$ is a common index divisor of $K$ for $p=2,3$. 
Their method is based on the calculation of the $p$-index form of $K$.
In  \cite{Ga21}, for a sextic number field $K$ defined by a trinomial $F(x)=x^6+ax^3+b\in \Z[x]$, Ga\'al  calculated all possible generators of power integral bases of $K$. In \cite{jnt6}, El Fadil extended Ga\'al's  studies  by providing some cases where $K$ is not monogenic. In \cite{BF}, Ben Yakkou and El Fadil gave  sufficient conditions on coefficients of a trinomial which guarantee the non-monogenity of the  number field defined by such a trinomial. Also in  \cite{com5}, for every prime integer $p$, El Fadil  gave necessary and sufficient conditions  on $a$ and $b$ which characterize when  $p$ is a common index divisor of $K$, where $K$ is a number field defined by  an irreducible trinomial $F(x)=x^5+ax^2+b\in \Z[x]$. 
Recall that the index of a  number field $K$ is defined by $i(K)=GCD((\mathbb{Z}_K:\mathbb{Z}[\theta])\mid K=\mathbb{Q}(\theta) \mbox{ and } \theta\in \Z_K)$. A rational prime integer $p$ dividing $i(K)$ is called a prime common index divisor of $K$. If $\mathbb{Z}_K$ has a power integral basis, then $i(K)=1$. Therefore a field having a prime common index divisor is not monogenic.
In this paper, for certain number  fields  $K= \Q(\th)$ generated by a   root  $\th $ of a non monogenic  irreducible   trinomial   $F(x)= x^n +ax^m+b$,
 we  give some sufficient conditions on $a$ and $b$, which guarantee the non-monogenity  of $K$.  
Our proposed results extend the special  case when $m=1$, which  has been previously studied by Ben Yakkou and El fadil studied  \cite{BF}.  
\section{Main Results}
 Throughout this section $K=\Q(\th)$ is a number field generated  by a complex  root  $\th$ of a monic irreducible trinomial of  type  $F(x) = x^n + ax^m +b \in \Z[x]$ with $m<n$ two natural numbers . 
 Let $p$ be  a natural prime integer, $\F_p$ the finite field with $p$ elements. For every $t \in \Z$, let $\nu_p(t)$ be the $p$-adic valuation of $t$  and  $t_p=\frac{t}{p^{\nu_p(t)}}$. For any $ t\in \Z$ and positive rational integers $f$, $s$, we  denote by $N_p(f)$ the number of  monic irreducible polynomials of degree $f$ in $\F_p[x]$ and $N_p(f,s,t)$ the number of monic irreducible factors of degree $f$ of the polynomial $x^s+\overline{t}$ in $\F_p[x]$.
 In this paper, we give some  sufficient conditions  on $a$, $b$, $n$, and $m$ for which $i(K)>1$, and hence $K$ is  not monogenic.  
\begin{lemma}\label{dn1}
	Let  $p$ be a natural prime integer which divides
	 both $a$ and $n$, and  does not divide $b$.  Let $\mu = \nu_p(a)$,  $\nu =\nu_p(b^{p-1}-1)$, and  $n = s\cdot p^r$ with $p \nmid s$. If for some positive integer $f$, one of the following conditions holds:
\begin{enumerate} 
\item $\mu < \min\{ \nu, r+1\}$ and $N_p(f)<\mu N_p(f,s,b)$, 
\item  $\nu < \min\{ \mu, r+1\}$ and $N_p(f)<\nu N_p(f,s,b)$, 
\item  $ r+1  \le \min\{ \nu, \mu\}$ and $N_p(f)<(r+1) N_p(f,s,b)$, 
\end{enumerate}	
 	then  $p$ divides $i(K)$. In particular $K$ is not monogenic.
\end{lemma}
 The following  theorem gives  an  infinite family  of non-monogenic   number fields defined by trinomials of degrees $2^k \cdot 3^r$.

\begin{theorem}\label{cordn1}
Let $K$ be a number field generated by a root $\al$ of  an irreducible trinomial $F(x)=x^{2^k \cdot 3^r}+ax^m+b\in \Z[x]$ with  $k$, $r$, and $m$ three natural integers such that $m+1\le  2^k \cdot 3^r$. If one of the following conditions holds:
\begin{enumerate}
\item
$k \neq 2$, $\nu_3(a)\ge 2 \mbox{ and } \, b \equiv -1 \md {9}$,
\item
 $k = 2$, $\nu_3(a)\ge 2 \mbox{ and } \, b \equiv \pm 1 \md {9}$,
\end{enumerate}
then $3$ divides $i(K)$. \\
In particular, any one of these conditions guarantees the non monogenity of $K$.
\end{theorem}
\begin{rem}
Under the hypotheses of Theorem \ref{cordn1}, if  $\nu_3(a)=1$ or  $b\not\in\{ 1,-1\}\md{9}$, then $3$ does not divide $(\Z_K:\Z[\al])$, and so $3$ does not divide $i(K)$.
\end{rem}
  The following  theorem gives  an  infinite family  of non-monogenic  sextic number fields defined by trinomials.
\begin{theorem}\label{d6}
Let $K=\Q(\th)$ be a sextic	 number field generated by a  root 	 of a monic irreducible trinomial $x^6+ax^m+b$  with   $m\le 5$ a natural integer. Then the following statements hold:
\begin{enumerate}
\item
  If $a\equiv 0 \md 9$ and $b \equiv -1 \md 9$, then $3$ divides $i(K)$.
\item  
 If $a\equiv 0 \md 8$ and $b \equiv -1 \md 4$, then  $2$ divides $i(K)$.
\end{enumerate}
 In particular, if one of these conditions holds, then $K$ is not monogenic.
\end{theorem}
\begin{rem}
Theorem \ref{d6}  extends the results given by Jakhar and Kumar  in \cite{A6}, where $m=1$ has been previously studied. 
\end{rem}
\section{Preliminaries}
	In order to   prove our main Theorems,  we recall some fundamental techniques on prime ideal factorization and calculation of index. 
  Let  $\ol{F(x)}=\prod_{i=1}^r \ol{\ph_i(x)}^{l_i}$ in  $\F_p[x]$ be the factorization of $\ol{F(x)}$ into powers of monic irreducible coprime polynomials of $\F_p[x]$. Recall that  a theorem of  Dedekind says that if 
 $$ p \mbox{  does not divide the index } (\Z_K:\Z[\al]), \mbox{ then } p\Z_K=\prod_{i=1}^r \p_i^{l_i}, \mbox{ where every } \p_i=p\Z_K+\phi_i(\al)\Z_K$$  and the residue degree of $\p_i$ is $f(\p_i)={\mbox{deg}}(\phi_i)$ ( see \cite[ Chapter I, Proposition 8.3]{Neu}).
 In order to apply this  theorem in an effective way,
  in $1878$, Dedekind  showed  the well known Dedekind's criterion which allows to test whether  $p$ divides or not  $(\Z_K: \Z[\al])$, see for instance \cite[Theorem 6.1.4]{Co} and \cite{R}.
 When Dedekind's criterion fails, that is,  $p$ divides the index $(\Z_K:\Z[\th])$ for every primitive element $\th\in\Z_K$ of $K$,  then for such primes and number fields, it is not possible to obtain the prime ideal factorization of $p\Z_K$ by {Dedekind's theorem}.
In 1928, Ore developed   an alternative approach
for obtaining the index $(\Z_K:\Z[\alpha])$, the
absolute discriminant, and the prime ideal factorization of the rational primes in
a number field $K$, under some conditions of $p$-regularity of $F(x)$, by using Newton polygons (see for instance {\cite{EMN, MN, O}}). 
	Now we recall  some fundamental {facts about   Newton polygons}. For more details, we refer to \cite{El} and \cite{GMN}.\\
	
	For any prime integer $p$ and for any monic polynomial 
	$\phi\in
	\mathbb{Z}[x]$  whose reduction is irreducible  in
	$\mathbb{F}_p[x]$, let $\mathbb{F}_{\phi}$ be the finite field $\mathbb{F}_p[x]/(\overline{\phi})$. For any monic polynomial $F(x)\in \mathbb{Z}[x]$, upon to the Euclidean division by successive powers of $\phi$, we expand $F(x)$ as $F(x)=a_0(x)+a_1(x)\phi(x)+\cdots+a_l(x)\phi(x)^l$, called the $\phi$-expansion of $F(x)$ (for every $i=0,\dots,l,\,\mbox{deg}(a_i(x))<\mbox{deg}(\phi))$. To any coefficient $a_i(x)\neq 0$, we attach the $p$-adic value $u_i=\nu_p(a_i(x))\in \mathbb{Z}$. The $\phi$-Newton polygon of $F(x)$ with respect to $p$, is the lower boundary convex envelope of the set of points $\{(i,u_i),\, a_i(x)\neq 0\}$ in the Euclidean plane, which we denote by $\nph{F}$. The $\phi$-Newton polygon of $F$, is the process of joining the obtained edges  $S_1,\dots,S_r$ ordered by   increasing slopes, which  can be expressed as $\nph{F}=S_1+\dots + S_r$.  The principal $\phi$-Newton polygon of ${F}$,
	denoted $\npp{F}$, is the part of the  polygon $\nph{F}$, which is  determined by joining all sides of negative  slopes.
For every side $S$ of $\npp{F}$, { the length of $S$, denoted $l(S)$, is the length of its projection to the $x$-axis and  its height, denoted $H(S)$, is the length of its projection to the $y$-axis}. {Let $d=$GCD$(l(S), H(S))$, called the  degree of $S$ and $e=\frac{l(S)}{d}$ called the ramification index of $S$.
			For every side $S$ of {$\npp{F}$}, with initial point $(s, u_s)$ and length $l$, and for every 
	$i=0, \dots,l$, we attach   the following
	{{\it residue coefficient} $c_i\in\mathbb{F}_{\phi}$ as follows:
		$$c_{i}=
		\left
		\{\begin{array}{ll} 0,& \mbox{ if } (s+i,{\it u_{s+i}}) \mbox{ lies strictly
			above } S,\\
		\left(\dfrac{a_{s+i}(x)}{p^{{\it u_{s+i}}}}\right)
		\,\,
		\mod{(p,\phi(x))},&\mbox{ if }(s+i,{\it u_{s+i}}) \mbox{ lies on }S.
		\end{array}
		\right.$$
		where $(p,\phi(x))$ is the maximal ideal of $\mathbb{Z}[x]$ generated by $p$ and $\phi$.
	Let $-\lambda=-h/e$ be the slope of $S$, where  $h$ and $e$ are two positive coprime integers. Then  $d=l/e$ is the degree of $S$. Since 
	the points  with integer coordinates lying{ on} $S$ are exactly $${(s,u_s),(s+e,u_{s}-h),\cdots, (s+de,u_{s}-dh)},$$ if $i$ is not a multiple of $e$, then 
	$(s+i, u_{s+i})$ does not lie on $S$, and so $c_i=0$. Let
	{$$R_{\lambda}(F)(y)=t_dy^d+t_{d-1}y^{d-1}+\cdots+t_{1}y+t_{0}\in\mathbb{F}_{\phi}[y],$$}} called  
the residual polynomial of $F(x)$ associated to the side $S$, where for every $i=0,\dots,d$,  $t_i=c_{ie}$.
\begin{rem}
Notice that as $(s,u_s)$ and $(s+de, u_s-dh)$ lie on $S$, we have $t_dt_0\neq 0$ in $\fph$. Thus $R_{\lambda}(F)(y)$ is of degree $d$ and $y$ does not divide $R_{\lambda}(F)(y)$.
\end{rem}
Let $\npp{F}=S_1+\dots + S_t$ be the principal $\phi$-Newton polygon of $F$ with respect to $p$.\\
We say that $F$ is a $\phi$-regular polynomial with respect to $p$, if  $R_{\lambda_i}(y)$ is square free in $\mathbb{F}_{\phi}[y]$ for every  $i=1,\dots,t$. 
The polynomial $F$ is said to be  $p$-regular  if $\overline{F(x)}=\prod_{i=1}^r\overline{\phi_i}^{l_i}$ for some monic polynomials $\phi_1,\dots,\phi_r$ of $\mathbb{Z}[x]$ such that $\overline{\phi_1},\dots,\overline{\phi_r}$ are irreducible pairwise coprime polynomials in $\mathbb{F}_p[x]$ and    $F(x)$ is  a $\phi_i$-regular polynomial with respect to $p$ for every $i=1,\dots,r$.
\smallskip
	The  theorem of Ore plays  a  key role for proving our main Theorems.\\
	Let $\phi\in\mathbb{Z}[x]$ be a monic polynomial, with $\overline{\phi(x)}$ is irreducible in $\mathbb{F}_p[x]$. As defined in \cite[Def. 1.3]{EMN},   the $\phi$-index of $F(x)$, denoted by $ind_{\phi}(F)$, is  deg$(\phi)$ times the number of points with natural integer coordinates that lie below or on the polygon $\npp{F}$, strictly above the horizontal axis,{ and strictly beyond the vertical axis} (see $Figure\ 1$).\\	
	\begin{figure}[htbp] 
		\centering
		\begin{tikzpicture}[x=1cm,y=0.5cm]
		\draw[latex-latex] (0,6) -- (0,0) -- (10,0) ;
		\draw[thick] (0,0) -- (-0.5,0);
		\draw[thick] (0,0) -- (0,-0.5);
		\node at (0,0) [below left,blue]{\footnotesize $0$};
		\draw[thick] plot coordinates{(0,5) (1,3) (5,1) (9,0)};
		\draw[thick, only marks, mark=x] plot coordinates{(1,1) (1,2) (1,3) (2,1)(2,2)     (3,1)  (3,2)  (4,1)(5,1)  };
		\node at (0.5,4.2) [above  ,blue]{\footnotesize $S_{1}$};
		\node at (3,2.2) [above   ,blue]{\footnotesize $S_{2}$};
		\node at (7,0.5) [above   ,blue]{\footnotesize $S_{3}$};
		\end{tikzpicture}
		\caption{\large  $\npp{F}$.}
	\end{figure}
	In the example of $Figure\ 1$, $ind_\phi(F)=9\times$deg$(\phi)$.\\
	\smallskip
	
	Now assume that $\overline{F(x)}=\prod_{i=1}^r\overline{\phi_i}^{l_i}$ is the factorization of $\overline{F(x)}$ in $\mathbb{F}_p[x]$, where every $\phi_i\in\mathbb{Z}[x]$ is a monic polynomial, with $\overline{\phi_1},\dots,\overline{\phi_r}$ are irreducible pairwise coprime polynomials over $\mathbb{F}_p$.
	For every $i=1,\dots,r$, let  $N_{\phi_i}^+(F)=S_{i1}+\dots+S_{ir_i}$ be the principal  $\phi_i$-Newton polygon of $F$ with respect to $p$. For every $j=1,\dots, r_i$,  let $R_{\lambda_{ij}}(F)(y)=\prod_{s=1}^{s_{ij}}\psi_{ijs}^{a_{ijs}}(y)$ be the factorization of $R_{\lambda_{ij}}(F)(y)$ in $\mathbb{F}_{\phi_i}[y]$. 
	Then we have the following  theorem of index of Ore:
	\begin{theorem}\label{ore}$($Theorem of Ore$)$ 
		\begin{enumerate}
			\item 
			$$\nu_p((\mathbb{Z}_K:\mathbb{Z}[\alpha]))\geq\sum_{i=1}^{r}ind_{\phi_i}(F).$$ 
			The equality holds if $F(x) \text{ is }p$-regular.
			\item 
			If $F(x) \text{ is }p$-regular, then
			$$p\mathbb{Z}_K=\prod_{i=1}^r\prod_{j=1}^{r_i}\prod_{s=1}^{s_{ij}}\mathfrak{p}_{ijs}^{e_{ij}}$$
is the factorization of $p\Z_K$ into product of powers of prime ideals of $\Z_K$, 
			where $e_{ij}$ is the ramification degree of the side $S_{ij}$ and $f_{ijs}=\deg(\phi_i)\times \deg(\psi_{ijs})$ is the residue degree of $\mathfrak{p}_{ijs}$ over $p$ for every $(i,j,s)$.
		\end{enumerate}
	\end{theorem}
\begin{cor}\label{cor1}
	Under the hypotheses above  Theorem \ref{ore}, if for every $i=1,\dots,r,\,l_i=1\text{ or }N_{\phi}^+(F)=S_i$ has a single side of height $1$, then $\nu_p((\mathbb{Z}_K:\mathbb{Z}[\alpha]))=0$.
\end{cor}
  In \cite{GMN}, Gu\`ardia, Montes, and  Nart introduced  the notion of admissible $\phi$-expansion used in order to treat some special cases when the $\phi$-expansion is hard to calculate. Let
\begin{equation}
\label{eq1}
F(x)=\sum_{i=0}^nA_i'(x)\phi(x)^i,\quad A_i'(x)\in \mathbb{Z}[x],
\end{equation}
be a $\phi$-expansion of $F(x)$, not necessarily the $\phi$-expansion (deg$(A_i')$ is not necessarily less than deg$(\ph)$). Take $u_i'=\nu_p(A_i'(x))$, for all $i=0,\dots, n$, and let $N'$ be the lower boundary of the convex envelope of the set of points $\{(i,u_i')\,\mid\,0\leq i\leq n,\,u_i'\neq\infty\}$ and $N'^+$ its principal part. To any $i=0,\dots,n$, we attach the residue coefficient as follows:
$$
c_i'=\left\{
\begin{array}{ll}
0,&\text{if }(i,u_i')\text{ lies above }N',\\
\left(\frac{A_i'(x)}{p^{u_i'}}\right)(\mod(p,\phi(x))),& \text{if }(i,u_i')\text{ lies on }N'.
\end{array}
\right.
$$
Likewise, for any side $S$ of $N'^+$, we can define the residual polynomial attached to $S$ and denoted $R_{\lambda}'(F)(y)\,($similar to the residual polynomial $R_{\lambda}(F)(y)$ from the $\phi$-expansion$)$. We say that the $\phi$-expansion $(\ref{eq1})$ is admissible if $c_i'\neq 0$ for each abscissa $i$ of a vertex of $N'$. For more details, we refer to \cite{GMN}.
\begin{lemma}{$($\cite[Lemma 1.12]{GMN}$)$}\\
	If a $\phi$-expansion of $F(x)$ is admissible, then $N'^+=\npp{F}$ and $c_i'=c_i$. In particular, for any side $S$ of $N'^+$ we have $R_{\lambda}'(F)(y)=R_{\lambda}(F)(y)$ up to {multiplication} by a nonzero coefficient of $\fph$.
\end{lemma}
\section{Proofs of main results}  
The following lemma is useful for proving Lemma \ref{NP}.
\begin{lemma} \label{binomial} 
	Let $p$ be a rational prime integer and $r$  a positive integer. Then\begin{eqnarray*}
	\nu_p\left(\binom{p^r}{j}\right)  =  r - \nu_p(j) 
	\end{eqnarray*}for any integer $j= 1,\dots,p^r-1 $. 
\end{lemma}
The following lemma plays a key role in the proof of Theorem \ref{dn1}; it allows to determine the $\ph$-Newton polygon of $F(x)$. It extends \cite[Lemma 4.1]{EC}, which allows to determine  $\ph$-Newton polygons for polynomials of type $x^n-m$.
\begin{lemma}\label{NP}
Let $F(x)=x^n+ax^m+b \in \Z[x]$ be  polynomial and $p$ an odd  rational prime integer which divides both $n$ and $a$, and  does not divide $b$. Let $n=p^rs$  with $p$ does not divide $s$. Set  $\mu=\nu_p(a)$,  $\nu=\nu_p(b^{p-1}-1)$, and $g=\mbox{min}(\mu,\nu)$. Then $\ol{F(x)}=\ol{(x^s+b)}^{p^r}$. Let  $\ph\in \Z[x]$ be a monic polynomial, whose reduction modulo $p$ divides $\ol{F(x)}$.  
 \begin{enumerate}
 \item
 If $g\le r$, then $\npp{F}$ has at least $g$ sides joining  the  points $\{(0,v)\}\cup \{(p^j,r-j), \, j=0,\dots,g-1\}$  for some integer $V\ge g$.
 \item
 If  $g\ge r+1$, then $\npp{F}$ has $r+1$ sides joining  the  points $\{(0,V)\}\cup \{(p^j,r-j), \, j=0,\dots,r-1\}$ for some integer $V\ge r+1$.
 \end{enumerate}
\end{lemma}
  \begin{proof}
Since $\ol{F(x)}=(x^s+b)^{p^r}$ in $\F_p[x]$, we conclude that $\ol{\ph}$ divides $\ol{x^s+b}$. Let $x^s+b=\ph(x)Q(x)+pR(x)$ for some $R\in \Z[x]$. Then $F(x)=(x^s+b-b)^{p^r}+ax^m+b=(\ph(x)Q(x)+pR(x)-b)^{p^r}+ax^m+b=\displaystyle\sum_{j=1}^{p^r}\binom{p^r}{j}(pR(x)-b)^{{p^r}-j}Q^j\phi^j(x)+(pR(x)-b)^{p^r}+ax^m+b=\displaystyle\sum_{j=1}^{p^r}\binom{p^r}{j}(pR(x)-b)^{{p^r}-j}Q^j\phi^j(x)+p^{r+1}H(x)-b^{p^r}+ax^m+b$ for some $H\in \Z[x]$.
Let $p^{r+1}H(x)+ax^m-b^{p^r}+b=\displaystyle\sum_{k\ge 0}r_j\phi^j(x)$ be the $\ph$-expansion.  Then
$F(x)=\dots+\displaystyle\sum_{j=1}^{p^r}A_j(x)\phi^j(x)+r_0$ is the $\ph$-expansion of $F(x)$ with $A_j(x)=\binom{p^r}{j}((pR(x)-b)^{{p^r}-j}Q^j+r_j)$ for every $j=1,\dots p^r$.  Since $x^s+b$ is separable over $\F_p$, $\ol{\ph}$ does not divide $\ol{Q(x)}$. Hence $\nu_p(A_j)=\nu_p(\binom{p^r}{j}((pR(x)-b)^{{p^r}-j}Q^j+r_j))=\nu_p(\binom{p^r}{j})$ for every $j=1,\dots,p^r-1$.
Also since for every
$j=0,\dots,p^r$, $\nu_p(r_j)\ge \mbox{min}(r+1,g)$, we conclude that:
 \begin{figure}[htbp] 
	\centering
	\begin{tikzpicture}[x=0.5cm,y=0.6cm]
	\draw[latex-latex] (0,6.5) -- (0,0) -- (29,0) ;
	\draw[thick] (0,0) -- (-0.5,0);
	\draw[thick] (0,0) -- (0,-0.5); 
	\draw[thick,red] (3,-2pt) -- (3,2pt);
	\draw[thick,red] (6,-2pt) -- (6,2pt);
	\draw[thick,red] (11,-2pt) -- (11,2pt);
	\draw[thick,red] (18,-2pt) -- (18,2pt);
	\draw[thick,red] (28,-2pt) -- (28,2pt);
	\draw[thick,red] (-2pt,1) -- (2pt,1);
	\draw[thick,red] (-2pt,2) -- (2pt,2);
	\draw[thick,red] (-2pt,3) -- (2pt,3);
	\draw[thick,red] (-2pt,5) -- (2pt,5);
	\draw[thick,red] (-2pt,5.5) -- (2pt,5.5);
		\draw[thick,red] (-2pt,6) -- (2pt,6);
	\node at (0,0) [below left,blue]{\footnotesize  $0$};
	\node at (3,0) [below ,blue]{\footnotesize $p^{r-x+1}$};
	\node at (6,0) [below ,blue]{\footnotesize  $p^{r-x+2}$};
	\node at (11,0) [below ,blue]{\footnotesize  $p^{r-2}$};
	\node at (18,0) [below ,blue]{\footnotesize  $p^{r-1}$};
	\node at (28,0) [below ,blue]{\footnotesize  $p^{r}$};
	\node at (0,1) [left ,blue]{\footnotesize  $1$};
	\node at (0,2) [left ,blue]{\footnotesize  $2$};
	\node at (0,5) [left ,blue]{\footnotesize  $v$};
	\draw[thick,mark=*] plot coordinates{(0,5) (3,3.7)};
	\draw[thick,mark=*] plot coordinates{(3,3.7) (6,3)};
	\draw[thick,mark=*] plot coordinates{(11,2) (18,1)};
	\draw[thick,mark=*] plot coordinates{(28,0) (18,1)};
	\draw[thick, dashed] plot coordinates{(6.5,2.9) (10.5,2.1) };
	\node at (0.5,4.5) [above right  ,blue]{\footnotesize  $S_{1}$};
	\node at (4.5,3.2) [above right  ,blue]{\footnotesize  $S_{2}$};
	\end{tikzpicture}
	\caption{\large  $\npp{F}$ with respect to $p$ when  $V	\ge \mbox{min}(g, r+1)$.\hspace{5cm}}
\end{figure}

\begin{enumerate}
\item
If $g\le r$, then   for every $j=1,\dots,p^{g}$, $\nu_p(A_j)=\nu_p(\binom{p^r}{j})=r-\nu_p(j)$, and so $\npp{F}$ is has at least $g$ sides joining  the  points $\{(0,V)\}\cup \{(p^j,r-j), \, j=0,\dots,g-1\}$ for some $V\ge g$.
 \item
 If  $g\ge r+1$, then for every $j=1,\dots,p^{r}$, $\nu_p(A_j)=\nu_p(\binom{p^r}{j})=r-\nu_p(j)$, and so  $\npp{F}$ has $r+1$ sides joining  the  points $\{(0,V)\}\cup \{(p^j,r-j), \, j=0,\dots,r-1\}$ for some integer $V\ge r+1$.
\end{enumerate}
 \end{proof}

The following Lemma characterizes the prime divisors of $K$:
\begin{lemma} \label{comindex}
	Let  $p$ be a  rational prime integer and $K$  a number field. For every positive integer $f$, let $P_p(f)$  be the number of distinct prime ideals of $\Z_K$ lying above $p$ with residue degree $f$ and $N_p(f)$ the number of monic irreducible polynomial of degree $f$ in $\F_p[x]$. If $ P_p(f) > N_p(f)$ for some positive integer $f$, then $p$ is a prime common index divisor of $K$.
\end{lemma}

	\begin{rem}\label{remore}
In order to prove  our theorems, we don't need to determine the factorization of $p\Z_K$ explicitly. But according to Lemma \ref{comindex}, we need only to show that $P_p(f)>{N}_p(f)$ for an adequate positive integer  $f$. So, in practice, the second \ point of Theorem \ref{ore}, could  be  replaced by the following:
If  $l_i=1$ or $d_{ij}=1$ or $a_{ijs}=1$ for some $(i,j,s)$ according to the notations of   Theorem \ref{ore}, then $\psi_{ijs}$ provides  a prime ideal $\p_{ijs}$ of $\Z_K$ lying above $p$ with residue degree  $f_{ijs}=$deg$(\phi_i)\times t_{ijs}$, where  $t_{ijs}=$deg$(\psi_{ijs})$ and $p\Z_K=\p_{ijs}^{e_ij}I$, where the factorization of the ideal $I$ can be derived from the other factors of each residual polynomials of $F(x)$.
	\end{rem}
	
\begin{proof}[Proof of Lemma \ref{dn1}]
 In all cases, we prove that $K$ is not monogenic by showing that $p$ divides $i(K)$. For this reason, in view of Lemma \ref{comindex}, it is sufficient to show that the prime ideal factorization of $p\Z_K$ satisfies the inequality $ P_p(f) > N_p(f)$ for an adequate positive integer $f$. Let $\ph$ be a monic polynomial with $\ol{\ph}$ divides $\ol{F(x)}$ in $\F_p[x]$. Let $g=\mbox{min}(\mu,\nu)$.
\begin{enumerate}
\item 
 If $g\le r$, then by Lemma \ref{NP},  	$\npp{(F)} $ has $g$ sides of degree $1$ each. 	Thus each side of $\npp{(F)} $ provides a unique prime ideal of $\Z_K$ lying above $p$ with residue degree $1$ each. Hence according to Remark \ref{remore}, the irreducible factor $\ol{\ph}$ of $\ol{F(x)}$ provides  at least $g$ prime ideals of $\Z_K$ lying above $p$ with residue degree $1$ each.
In total there  are at least $gN_p(f,s,b)$ prime ideals of $\Z_K$ lying above $p$ with residue degree $1$ each. It follows that if 
$N_p(f)<g N_p(f,s,b)$, then $p $ divides $ i(K)$, and consequently $K$ is not monogenic.
\item
If  $ r+1  \le g$, then by Lemma \ref{NP},  	$\npp{(F)} $ has $r+1$ sides of degree $1$ each. 	Thus  each side of $\npp{(F)} $ provides a unique prime ideal of $\Z_K$ lying above $p$ with residue degree $1$ each. Hence  according to Remark \ref{remore}, the irreducible factor $\ol{\ph}$ of $\ol{F(x)}$ provides   $r+1$ prime ideals of $\Z_K$ lying above $p$ with residue degree $1$ each.
In total there  are at least $(r+1)N_p(f,s,b)$ prime ideals of $\Z_K$ lying above $p$ with residue degree $1$ each. It follows that if 
$N_p(f)<(r+1) N_p(f,s,b)$, then $p $ divides $ i(K)$, and consequently $K$ is not monogenic.
\end{enumerate}
\end{proof}
\begin{proof}[Proof of Theorem \ref{cordn1}]
\begin{enumerate}
\item 
If $b\equiv -1\md9$, then $\ol{F(x)}=(x^{2^k}-1)^{3^r}(x-1)^{3^r}(x+1)^{3^r}H(x)^{3^r}$ in $\F_3[x]$. According to the  notations of Theorem \ref{dn1}, we have $s=2^k$, $N_3(1,s,b)\ge 2$ and $N_3(1)=3$. Since $r+1$, $\mu,$ and $\nu$ are all greater or equal $2$, we have  $N_3(1)=3<4\le 2N_3(2,4,1)$.  By Theorem \ref{dn1}, we  conclude that
$3$ divides  $i(K)$.
\item
The first point covers the case when $b\equiv -1\md9$. For $b\equiv 1\md9$, we have $x^s+1=x^4+1=(x^2-x-1)(x^2+x-1)$  in $\F_3[x]$. Since the unique irreducible polynomial of degree $2$ in $\F_3[x]$ are $x^2+1$, $x^2-x-1$, and $x^2+x-1$, we conclude that $N_3(2)=3$. Thus if $k=2$, then $N_3(2)=3<4\le 2N_3(2,4,-1)$. Thus by Theorem \ref{dn1},  $3$ divides  $i(K)$.
\end{enumerate}
\end{proof}
\begin{proof}[Proof of Theorem \ref{d6}]
The first point is a particular case of Theorem \ref{cordn1}. For the second one,
 since $a\equiv 0\md 8$ and $b\equiv -1 \md 4$,  $\ol{F(x)}= (\ol{\ph_1(x)\ph_2(x)})^2$ in $\F_2[x]$, where  $\ph_1(x)=x-1$ and $\ph_2=x^2+x+1$. 
Let		$F(x)= \ph_2(x)^3-3x\ph_2(x)^2+(-2+2x)\ph_2+ax^m+1+b$ and $ax^m=\sum_{i=0}^2r_i(x)\ph_2(x)^i$. Then $F(x)= \ph_2(x)^3+(-3x+r_2(x))\ph_2(x)^2+(-2+2x+r_1(x))\ph_2+1+b+r_0(x)$.
Since $\nu_2(ax^m)\ge 3$, we conclude that if   $\nu_2(b+1)\ge 3$, then  $N_{\ph_2}^+(F)=S_1+S_2$ has two sides joining $(0,v)$, $(1,1)$, and $(2,0)$ with   $v=\nu_2(b+1)\ge 3$. So each side $S_i$ is of degree $1$, and  so $\ph_2$ provides  two prime ideals of $\Z_K$ lying above $2$ with residue degree $2$ each. Hence $2$ divides $i(K)$. If   $\nu_2(b+1)= 2$, then  $N_{\ph_2}^+(F)=S$ has a single  side joining $(0,2)$, $(1,1)$, and $(2,0)$ with $R_{\lambda}(F)(y)=xy^2+(x+1)y+1=x(y-1)(y-x-1)$ its attached residual polynomial of $F(x)$.  Hence $\ph_2$ provides  two prime ideals of $\Z_K$ lying above $2$ with residue degree $2$ each. Therefore $2$ divides $i(K)$. 
\end{proof}
\section{Examples}
Let $F(x) \in \Z[x]$ be a	 monic irreducible polynomial  and $K$ a number field generated by a  root $\al$ of $F(x)$.

\begin{enumerate}
\item  
For $F(x)=x^{6}+72 x^m+35$,  since $72\equiv 0\md9$ and $35\equiv -1\md9$ by Theorem \ref{d6}, $3$ divides $i(K)$. Similarly  since $72\equiv 0\md8$ and $35\equiv -1\md4$, by Theorem \ref{d6}, $2$ divides $i(K)$. Thus $6$ divides $i(K)$. 
Since $\ol{F(x)}=\ph_1^3\ph_2^3$ in $\F_3[x]$ with $\ph_1=x-1$ and $\ph_2=x+1$, and  each  $\ph_i$ provides $3$ prime ideals of $\Z_K$ lying above $3$ with residue degree $1$ each, we conclude that there are $6$ prime ideals of  of $\Z_K$ lying above $3$ with residue degree $1$ each. By \cite{En},  $\nu_3(i(K))=3$.
For $\nu_2(i(K))$, we have $\ol{F(x)}= (\ol{\ph_1(x)\ph_2(x)})^2$ in $\F_2[x]$, where  $\ph_1(x)=x-1$ and $\ph_2=x^2+x+1$.
As it is shown in the proof of Theorem \ref{d6}, since $\nu_2(35+1)=2$, $\ph_1$ provides a unique prime ideal of $\Z_K$ lying above $2$ with residue degree $2$ and   $\ph_2$ provides  two prime ideals of $\Z_K$ lying above $2$ with residue degree $2$ each. Thus there are  three prime ideals of $\Z_K$ lying above $2$ with residue degree $2$ each.  By \cite{En}, we conclude that  $\nu_2(i(K))=2$.
For $\nu_5(i(K))$, if $m=1$, then $\ol{F(x)}=(x+2)^5x$ in $\F_5[x]$ and  if $m=5$, then $\ol{F(x)}=(x+2)x^5$ in $\F_5[x]$. Since $F(-2)=-55$, we have $\nu_5(F(-2))=1$ and $\nu_5(F(0))=1$. Thus by Dedekind's criterion $5$ divide $(\Z_K:\Z[\al])$. For $m\in\{2,3,4\}$, $\ol{F(x)}$ has at least an irreducible factor of degree at least $2$. Thus there are at most four prime ideals of $\Z_K$ lying above $5$. Therefore $5$ does not divide $i(K)$. By \cite{C} $\nu_p(i(K))=0$ for every prime integer $p\ge 7$.  Finally $i(K)=2^2\cdot 3^3=36$.
\item  
For $F(x)=x^{6}+72 x^m+71$,   since $72\equiv 0\md9$ and $71\equiv -1\md9$ by Theorem \ref{d6}, $3$ divides $i(K)$. Similarly  since $72\equiv 0\md8$ and $71\equiv -1\md4$, by Theorem \ref{d6}, $2$ divides $i(K)$. Thus $6$ divides $i(K)$.
Since $\ol{F(x)}=\ph_1^3\ph_2^3$ in $\F_3[x]$ with $\ph_1=x-1$ and $\ph_2=x+1$, and  each  $\ph_i$ provides $3$ prime ideals of $\Z_K$ lying above $3$ with residue degree $1$ each, we conclude from \cite{En}, that $\nu_3(i(K))=3$.
For $\nu_2(i(K))$, we have $\ol{F(x)}= (\ol{\ph_1(x)\ph_2(x)})^2$ in $\F_2[x]$, where  $\ph_1(x)=x-1$ and $\ph_2=x^2+x+1$.
As it is shown in the proof of Theorem \ref{d6}, since $\nu_2(71+1)=3$, $\ph_1$ provides two prime ideals of $\Z_K$ lying above $2$ with residue degree $1$ each  and   $\ph_2$ provides  two prime ideals of $\Z_K$ lying above $2$ with residue degree $2$ each.   By \cite{En}, we conclude that  $\nu_2(i(K))=2$.
For $\nu_5(i(K))$, if $m=1$, then $\ol{F(x)}=(x+2)^5x$ in $\F_5[x]$ and  if $m=5$, then $\ol{F(x)}=(x+2)x^5$ in $\F_5[x]$. Since $F(-2)=-55$, we have $\nu_5(F(-2))=1$ and $\nu_5(F(0))=1$. Thus by Dedekind's criterion $5$ divide $(\Z_K:\Z[\al])$. For $m\in\{2,3,4\}$, $\ol{F(x)}$ has at least an irreducible factor of degree at least $2$. Thus there are at most four prime ideals of $\Z_K$ lying above $5$. Therefore $5$ does not divide $i(K)$. By \cite{C} $\nu_p(i(K))=0$ for every prime integer $p\ge 7$.  Finally $i(K)=2^2\cdot 3^3=36$.
\end{enumerate}

\end{document}